\documentclass[reqno,12pt, oneside]{amsart}

\usepackage{amssymb}
\usepackage[hidelinks]{hyperref}
\usepackage[textsize=footnotesize]{todonotes}
\usepackage{tikz}
\usetikzlibrary{cd}
\usepackage{float}
\usepackage{xcolor}
\usepackage{todonotes}
\usepackage{framed}

\theoremstyle{plain}
\newtheorem{theorem}{Theorem}[section]
\newtheorem{lemma}[theorem]{Lemma}
\newtheorem{definition}[theorem]{Definition}

\newtheorem{proposition}[theorem]{Proposition}
\newtheorem{remark}[theorem]{Remark}
\newtheorem{example}[theorem]{Example}
\newtheorem{question}[theorem]{Question}

\DeclareMathOperator{\var}{var}

\newcommand{\R}{\mathbb R}

\newcommand{\N}{\mathbb N}

\newcommand{\cP}{\mathcal P}
\newcommand{\cM}{\mathcal M}
\newcommand{\cG}{\mathcal G}
\newcommand{\Fin}{\mathrm{Fin}}
\newcommand{\Exh}{\mathrm{Exh}}
\newcommand{\FIN}{\mathrm{FIN}}
\newcommand{\EXH}{\mathrm{EXH}}

\newcommand{\BV}{\mathrm{BV}}

\newcommand{\norm}[1]{\left\|#1\right\|}

\newcommand{\abs}[1]{\lvert#1\rvert}

\newcommand{\zdef}{{\mathrel{\mathop:}=}}

\newcommand{\lnorm}[1]{\left\|#1\right\|_{A BV}}

\newcommand{\ABV}{{ABV}}
\newcommand{\BBV}{{BBV}}

\title[Exh subspaces of submeasure BV spaces]{Exh subspaces of bounded variation spaces generated by submeasures}

\author[J.~Gulgowski]{Jacek Gulgowski}
\address[J.~Gulgowski]{University of Warsaw,
Center of Trustworthy AI for Life Sciences - International Research Agendas Programme, Warsaw, Poland}
\email{j.gulgowski@uw.edu.pl}

\author[A.~Kwela]{Adam Kwela}
\address[A.~Kwela]{Institute of Mathematics\\ Faculty of Mathematics\\ Physics and Informatics\\ University of Gda\'{n}sk\\ ul.~Wita  Stwosza 57\\ 80-308 Gda\'{n}sk\\ Poland}
\email{Adam.Kwela@ug.edu.pl}
\urladdr{https://mat.ug.edu.pl/~akwela}

\author[P.~Sawicki]{Paweł Sawicki}
\address[P.~Sawicki]{Institute of Mathematics\\ Faculty of Mathematics, Physics and Informatics\\ University of Gda\'{n}sk\\ ul. Wita Stwosza 57\\ 80-308 Gda\'{n}sk\\ Poland}
\email{p.a.sawicki@wp.pl}

\author[J.~Tryba]{Jacek Tryba}
\address[J.~Tryba]{Institute of Mathematics\\ Faculty of Mathematics, Physics and Informatics\\ University of Gda\'{n}sk\\ ul. Wita Stwosza 57\\ 80-308 Gda\'{n}sk\\ Poland}
\email{Jacek.Tryba@ug.edu.pl}



\begin{document}

\begin{abstract}
In the present paper we are building one more bridge between the theory of ideals on the set of natural numbers and the theory of spaces of functions of bounded variation. It is already known that $F_\sigma$ ideals (that is, ideals of the form $\Fin(\phi)$, where $\phi$ is a lower semicontinuous submeasure on $\mathbb{N}$) lead to a nice generalization of different concepts from the theory of functions of bounded variation (Waterman variation and Chanturya classes). In this paper we show that analytic P-ideals (that is, ideals of the form $\Exh(\phi)$, where $\phi$ is a lower semicontinuous submeasure on $\mathbb{N}$) lead to a generalization of spaces of functions continuous in $\Lambda$-variation. The properties of this generalization are investigated.
\end{abstract}

\maketitle

\section{Introduction}

In the recent paper by Borodulin-Nadzieja and Farkas \cite{FBN} the new concept of sequence spaces $\FIN(\phi)$ and $\EXH(\phi)$ generated by a submeasure $\phi$ defined on the set of natural numbers was introduced. This concept, on the one hand, appears to be an interesting bridge between the theory of Banach sequence spaces and the theory of ideals of countable sets. On the other hand, it shows its power of generalization: for example it is a generalization of combinatorial Banach spaces which are a very complex creatures on its own (see e.g. \cite{comb:zoo}). 

In the paper \cite{GKT} it was shown that the sequence spaces $\FIN(\phi)$ may be used to define the function spaces $\BV(\phi)$ of the bounded-variation type. Moreover, it was shown that the spaces $\BV(\phi)$ are actually unifying two different concepts appearing in the theory  of bounded-variation spaces: Waterman $\Lambda$-variation  (introduced in \cite{Wat1}) and Chanturia classes (introduced in \cite{Cha74}). When we started to work on the mentioned generalization we almost immediately realized that  the way how the sequence space $\EXH(\phi)$ is defined may be introduced in the universe of bounded-variation functions. Actually, when $\phi$ is a measure, with measures of singletons $\phi(\{n\})=a_n$ forming a Waterman sequence, the space generated according to $\EXH$-pattern corresponds to the space of functions known as {\em continuous in $\Lambda$-variation} (introduced by Waterman in \cite[Section 4]{Wat1976} and later studied by many authors -- for details we refer to Subsection \ref{sec:cont:LBV} below and the references mentioned therein). 

In the present paper we introduce the concept of spaces $\BV_0(\phi)$ of functions continuous in submeasure variation, which not only is a generalization of spaces of functions continuous in $\Lambda$-variation, but also gives nontrivial subspaces of Chanturia classes. Some interesting observations are made about these spaces and some interesting open problems are formulated. We show that this concept is highly nontrivial and the tools developed for Waterman spaces are not powerful enough to work in this more general case.    

\section{Notation and preliminaries}

\subsection{Basic notation}
The set of natural numbers will be denoted by $\N = \{1,2,...\}$ and we have $\N_0 =\N\cup\{0\} = \{0,1,2,...\}$. By ${\mathcal P}(\N)$ we will denote the family of all subsets of $\N$.

Let us denote the unit interval by $I=[0,1]$. All function spaces considered in this paper will consist of functions $x\colon I\to\R$. Moreover, by ${\mathcal P}_I$ we denote the set of all sequences  of nonoverlapping, closed  subintervals $(I_1, I_2, \ldots, I_N, \ldots)$ of $I$. Let us recall that the intervals are called {\em nonoverlapping} if for each pair of intervals their intersection has at most one element. The intervals may be {\em degenerate}, i.e. it may happen that $I_n$ consists of only one point. 

For a function $x\colon I\to\R$ and any interval $J=[a,b]\subset I$ we will write
\[
\abs{x(J)} = \abs{x([a,b])} = |x(b)-x(a)|.
\]
\subsection{Functions continuous in $\Lambda$-variation}\label{sec:cont:LBV}

In the theory of $\Lambda$-variation spaces the concept of functions continuous in $\Lambda$-variation appeared. We are going to remind this concept, but first let us recall the idea of $\Lambda$-variation of the function $x:I\to\R$. This concept was introduced by Waterman in \cite{Wat1}. Since then the functions of bounded $\Lambda$-variation were intensively studied by many authors -- for an overview we refer to \cite{ABM}.

Let us assume that $A=(a_n)_{n\in\N}$ is such non-increasing sequence of positive real numbers that $\sum_{n=1}^{\infty} a_n = +\infty$. We call such sequence \emph{a~Waterman sequence}. If additionally $\lim_{n\to +\infty}a_n = 0$, we say that the sequence $A$ is a \emph{proper Waterman sequence}.

 \begin{definition}
  Let $A = (a_n)_{n\in\mathbb{N}}$ be a Waterman sequence and let $x\colon I\to\mathbb{R}$. We say that $x$ is of bounded $A$-variation  if there exists a positive  constant $M$ such that for any sequence of nonoverlapping  subintervals $(I_1, I_2, I_3, \ldots) \in {\mathcal P}_I$, the following inequality holds
  \[
   \sum_{n=1}^{+\infty} a_n\abs{x(I_n)} \leq M.
  \]
  The supremum of the above sums, taken over the family ${\mathcal P}_I$ of all sequences of nonoverlapping subintervals of $I$, is called the $A$-variation of $x$  and it is denoted by $\var_A(x)$. 
 \end{definition}

The space of all  functions defined on the interval $I$ and  of bounded $A$-variation, endowed with the norm $\lnorm{x} \zdef \abs{x(0)} + \var_A(x)$ forms a Banach space $\ABV(I)$ (see \cite[Section 3]{Wat1976}) called the \emph{Waterman space}.

The spaces $\ABV(I)$, treated as linear spaces, are proper subspaces of the linear space $B(I)$ of all bounded functions $x\colon I\to\R$. We assume the space $B(I)$ is equipped with the standard supremum norm
\[
\norm{x}_\infty = \sup_{t\in I} |x(t)|.
\]
With this assumption each space $\ABV(I)$ is continuously embedded in $B(I)$.

Given a positive integer $m\in\N$, the omission of the first $m$
terms of the Waterman sequence $A=(a_n)$ gives us a new  Waterman sequence $a^{(m)}_n = a_{m+n}$ that will be denoted by $A_{(m)}$.

\begin{definition}\label{def:cont:Lambda:var}
    The function $x\colon I\to\R$ is said to be {\em continuous in $A$-variation, denoted by $x\in ABV_c(I)$, if 
    \[
    \lim_{m\to+\infty} \var_{A_{(m)}}(x) = 0.
    \]}
\end{definition}

The properties of these functions were investigated in \cite{Wisniowski-separability}. The most important observation related to such functions is given in Theorem \ref{thr1} below. But first we need to give a few new defintions. Let us start with denoting the subset of $\ABV(I)$ ($\ABV_c(I)$) consisting of all functions which are also continuous by $C\ABV(I)$ ($C\ABV_c(I)$). We will also need the following definition.

\begin{definition}[cf. {\cite[Theorem 3.1]{Wisniowski-separability}}]
If $A$ is a Waterman sequence, then
$$S_A=\limsup_{n\to\infty}\frac{\sum_{i=1}^{2n}a_i}{\sum_{i=1}^{n}a_i}$$
is called the Shao-Sablin index of $A$.
\end{definition}

\begin{theorem}[cf. \cite{Wisniowski-separability}]
\label{thr1}
For every proper Waterman sequence $A$, the following statements are equivalent:
\begin{enumerate}
\item the space $C\ABV(I)$ is separable;
\item $C\ABV_c(I)=C\ABV(I)$;
\item $\ABV_c(I)=\ABV(I)$;
\item $S_A<2.$
\end{enumerate}
\end{theorem}

\subsection{Bounded variation spaces generated by submeasures}

A function $\phi:\mathcal{P}(\N)\to[0,\infty]$ is called a \emph{submeasure} if $\phi(\emptyset)=0$, $\phi(\{n\})<\infty$ for every $n\in\N$, and 
\[
\phi(C)\leq\phi(C\cup D)\leq\phi(C)+\phi(D)
\]
for all $C,D\subseteq\N$. A submeasure $\phi$ is \emph{lower semicontinuous} (lsc, in short) if $\phi(C)=\lim_{n\to\infty}\phi(C\cap\{1,2,\ldots,n\})$ for each $C\subseteq\N$.

An lsc submeasure $\phi$ is \emph{non-pathological}, if
$$\phi(C)=\sup\{\mu(C):\ \mu\text{ is a measure such that }\mu\leq\phi\},$$
for all $C\subseteq\N$ (here $\mu\leq\phi$ means that $\mu(D)\leq\phi(D)$ for all $D\subseteq\N$). Not every lsc submeasure is non-pathological -- see \cite[Section~1.9]{farah-book}, \cite[Theorem~4.12]{ft-mazur}, \cite[Section~6.2]{marciszewski-sobota}, \cite[Section~3]{MezaPat} or \cite[Theorem~4.7]{tryba-gendensity} for such examples. 

Given an lsc submeasure $\phi$, let $\mathcal{M}_\phi$ be the family of all measures $\mu$ on $\N$ such that $\mu\leq\phi$. Then, by definition, $\phi$ is non-pathological if and only if $\phi(C)=\sup\{\mu(C): \mu\in\mathcal{M}_\phi\}$ for all $C\subseteq\N$.

For lsc submeasure $\phi$ we may naturally define two sequence spaces (cf. \cite{FBN}).

\begin{definition}\label{def:seq:spaces}
Let $\phi$ be a non-pathological lsc submeasure. Define a function $\hat\phi:\R^\N\to[0,\infty]$ by
\[
\hat\phi(x)=\sup\left\{\sum_{n\in\N}\mu(\{n\})|x_n|: \mu\in\mathcal{M}_\phi\right\}
\]
for all $x=(x_n)_{n\in\N}\in\R^\N$. Define also
\[
\FIN(\phi)=\left\{x\in\R^\N: \hat\phi(x)<\infty\right\};
\]
\[
\EXH(\phi)=\left\{x\in\R^\N: \lim_{n\to\infty}\hat\phi(x\cdot\chi_{\{n,n+1,\ldots\}})=0\right\}.
\]
\end{definition}

And now we are ready to define the submeasure-generated variation of a function $x\colon I\to\R$. 

\begin{definition}
Let $\phi$ be a non-pathological lsc submeasure. For $J=(J_n)\in\mathcal{P}_I$ denote by $x(J)$ the sequence $(|x(J_n)|)$. Define
\[
\BV(\phi)=\left\{x\in B(I): \sup_{J\in\mathcal{P}_I}\hat\phi(x(J))<\infty\right\}.
\]
\end{definition}

Let us denote
\[
\var_\phi(x) = \sup_{J\in\mathcal{P}_I}\hat\phi(x(J)).
\]

This concept was defined in \cite{GKT}, some other properties of this idea were later investigated in \cite{GKT2}. This definition has proven its power to generalize concepts so different as Waterman spaces and Chanturia classes. The definition of generalized variation related to submeasure $\phi$ is naturally related to the sequence space $\FIN(\phi)$. This, since the moment the idea was born, opened the question if a similar construction may be started with the sequence space $\EXH(\phi)$. We are going to define the class of subspaces of $\BV(\phi)$ which are the generalization of the spaces of functions continuous in $\Lambda$-variation as described in the subsection \ref{sec:cont:LBV}.

\section{Functions continuous in submeasure variation}

We will now try to redefine the concept of functions continuous in $A$-variation to the more general concept of  functions continuous in submeasure-generated variation.

Let us now define the sequence of submeasures $(\phi_{(m)})$ given by the formula
\begin{equation*}
\phi_{(m)}(C) = \phi(C \cap \{m+1,m+2,...\}).
\end{equation*}
for $C\subseteq \N$ and fixed $m\in\N_0$.

It is obvious that if $\phi$ is a non-pathological lsc submeasure then so is $\phi_{(m)}$, for any $m\in\N_0$.

\begin{definition}
Let $\phi$ be a non-pathological lsc submeasure and $x\colon I\to\R$. Then $x$ is said to be {\em continuous in $\phi$-variation, denoted by $x\in \BV_0(\phi)$, if 
    \[
    \lim_{m\to+\infty} \var_{\phi_{(m)}}(x) = 0.
    \]}
\end{definition}

The idea given above is a direct translation of the concept of functions continuous in $\Lambda$-variation -- this was given in Definition \ref{def:cont:Lambda:var} and there is a characterization for the space of such functions given in Theorem \ref{thr1} above.

Let us start with a simple observation.

\begin{remark}
\label{rem:M_phi_m}
Note that \[\mathcal{M}_{\phi_{(m)}}=\{\mu\in\mathcal{M}_{\phi}: \mu(\{1,2,\ldots,m\})=0\}\]
for every $m\in\N_0$. Indeed, if $\mu\in\mathcal{M}_{\phi_{(m)}}$ then $\mu(\{1,2,\ldots,m\})\leq\phi_{(m)}(\{1,2,\ldots,m\})=0$ and $\mu(C)\leq\phi_{(m)}(C)=\phi(C\cap\{m+1,m+2,\ldots\})\leq\phi(C)$ for all $C\subseteq\N$. On the other hand, if $\mu\in\mathcal{M}_{\phi}$ is such that $\mu(\{1,2,\ldots,m\})=0$ then $\mu(C)=\mu(C\cap\{m+1,m+2,\ldots\})\leq\phi(C\cap\{m+1,m+2,\ldots\})=\phi_{(m)}(C)$ for all $C\subseteq\N$. 
\end{remark}

\section{Equivalent definitions}

Recall that the sequence of submeasures $(\phi_{(m)})$ is given by the formula
\begin{equation*}
\phi_{(m)}(C) = \phi(C \cap \{m+1,m+2,...\}).
\end{equation*}
for $C\subseteq \N$ and fixed $m\in\N_0$.

An alternative definition would be:
Let us now define the sequence of submeasures $(\phi_{(m)}^a)$ given by the formula
\[
\phi_{(m)}^a(C) = \phi(C + m) = \phi(\{ n+m : n\in C\}).
\]
for $C\subseteq \N$ and $m\in\N_0$.

As we will see in Proposition \ref{prop:var_phi=var_phi^a}, both approaches lead to the same definition of $\BV_0(\phi)$.

The first attitude directly refers to the definition of $\EXH(\phi)$ (see Definition \ref{def:seq:spaces} above), because
\[
\hat\phi(x\cdot\chi_{\{m+1,n+2,\ldots\}})= \hat\phi_{(m)}(x).
\]
On the other hand, the second one is written in the spirit of the Definition \ref{def:cont:Lambda:var} of functions continuous in $A$-variation: when $\phi$ is a measure given by a Waterman sequence $A$ then the measure $\phi_{(m)}^a$ is given by the Waterman sequence $A_{(m)}$.

The difference between the two definitions may be expressed by the following observation:
\[
\phi_{(m)}(\{1,...,k\}) = \phi(\{m+1,...,k\}),
\]
\[
\phi_{(m)}^a(\{1,...,k\}) = \phi(\{m+1,...,m+k\}),
\]
for $k>m$. So the submeasure $\phi_{(m)}$ is supported on some subset of the set $\{m+1,m+2,...\}$ and it may never be supported on the entire $\N$, while the submeasure $\phi_{(m)}^a$ may be supported on the entire $\N$. The other difference is that the definition of $\phi_{(m)}$ seems to be {\em more stable}: let us take an infinite subset $C$ of even numbers together with a submeasure $\phi$ supported on the set of even numbers and such that for any even $n$ there is $\phi(\{n\})>0$. Then for each $m\in\N$ we have $\phi_{(m)}(C)>0$, while $\phi_{(m)}^a(C)=0$ for all odd $m$.
Note that for such a submeasure, if we take such sequence $x$ that $x_n=1$ for odd $n$ and $x_n=0$ for even $n$, then $x\in\FIN(\phi)$, $\lim_{m\to+\infty} \hat\phi_{(m)}(x) = 0$ while  $\limsup_{m\to+\infty} \hat\phi^a_{(m)}(x) = \lim_{m\to\infty} \phi(\N\setminus m)$, which can be infinite for e.g. $\phi(A)=\sum_{n\in A\cap 2\N}\frac{1}{n}$.

\begin{proposition}
\label{prop:var_phi=var_phi^a}
Let $\phi$ be a non-pathological lsc submeasure and $x\in B(I)$. Then
$x\in \BV_0(\phi)$ if and only if \[
    \lim_{m\to+\infty} \var_{\phi^a_{(m)}}(x) = 0.
    \]
\end{proposition}

\begin{proof}
    Let $m\in\mathbb{N}$. In order to prove that $\var_{\phi_{(m)}}(x)\leq \var_{\phi^a_{(m)}}(x)$, we will show that for every $J=(J_1,J_2,\ldots)\in\mathcal{P}_I$ and $\mu\in\mathcal{M}_{\phi_{(m)}}$ there are $\hat{J}=(\hat{J}_1,\hat{J}_2,\ldots)\in\mathcal{P}_I$ and $\hat{\mu}\in\mathcal{M}_{\phi^a_{(m)}}$ such that \[\sum_{n\in\mathbb{N}}\mu(\{n\})|x(J_n)|=\sum_{n\in\mathbb{N}}\hat{\mu}(\{n\})|x(\hat{J}_n)|.\]
    Fix $J=(J_1,J_2,\ldots)\in\mathcal{P}_I$ and $\mu\in\mathcal{M}_{\phi_{(m)}}$. Define $\hat{J}_n=J_{m+n}$ for all $n\in\mathbb{N}$. Let also $\hat{\mu}$ be the measure given by $\hat{\mu}(\{n\})=\mu(\{n+m\})$ for all $n\in\mathbb{N}$. Observe that $\hat{\mu}\in\mathcal{M}_{\phi^a_{(m)}}$, since\[
    \hat{\mu}(C)=\mu(C+m)\leq\phi(C+m)=\phi^a_{(m)}(C)
    \]
    for all $C\subseteq\mathbb{N}$. Moreover,
    \begin{equation*}
\begin{split}
\sum_{n\in\mathbb{N}}\mu(\{n\})|x(J_n)| & =\sum_{n>m}\mu(\{n\})|x(J_n)|=\sum_{k\in\mathbb{N}}\mu(\{k+m\})|x(J_{k+m})| \\
          & =\sum_{k\in\mathbb{N}}\hat{\mu}(\{k\})|x(\hat{J}_k)|.
\end{split}
\end{equation*}

    Now we will show that that $\var_{\phi_{(m)}}(x)\geq \var_{\phi^a_{(m)}}(x)$. Fix any $\hat{J}=(\hat{J}_1,\hat{J}_2,\ldots)\in\mathcal{P}_I$ and $\hat{\mu}\in\mathcal{M}_{\phi^a_{(m)}}$. Let $J_n=\hat{J}_{n-m}$ for $n>m$ and $J_n$ be degenerate for $n\leq m$. Similarly, define $\mu$ as the measure given by $\mu(\{n\})=\hat{\mu}(\{n-m\})$ for $n>m$ and $\mu(\{n\})=0$ for $n\leq m$. Observe that $\mu\in\mathcal{M}$, since for every $C\subseteq\mathbb{N}$ we have
    \[
    \mu(C)=\mu(C_m)=\hat{\mu}(C_m-m)\leq\phi^a_{(m)}(C_m-m)=\phi(C_m)\leq\phi(C),
    \]
    where $C_m=C\cap\{m+1,m+2,\ldots\}$. As $\mu(\{n\})=0$ for $n\leq m$, we obtain $\mu\in\mathcal{M}_{\phi_{(m)}}$ (see Remark \ref{rem:M_phi_m}). Moreover,
    \begin{equation*}
\begin{split}
\sum_{k\in\mathbb{N}}\hat{\mu}(\{k\})|x(\hat{J}_k)| & =\sum_{k\in\mathbb{N}}\mu(\{k+m\})|x(J_{k+m})|=\sum_{n>m}\mu(\{n\})|x(J_{n})| \\
          & =\sum_{n\in\mathbb{N}}\mu(\{n\})|x(J_{n})|.
\end{split}
\end{equation*}
\end{proof}

Now we move to the next equivalent way of defining $\BV_0(\phi)$. We will need the following concept, which is a generalization of classes defined in \cite[Definition 1]{PW1} as Wiener $\Lambda$-variation.

First let us fix $\delta>0$ and a function $x\colon I\to\R$. Then by  $\cP_I^\delta(x)$ let us denote the family of all sequences of nonoverlapping subintervals $[a,b]$ of $I$ such that $|x(b)-x(a)|\leq\delta$. Let us then denote

\[
\var_\phi^\delta(x) = \sup_{J\in\mathcal{P}_I^\delta(x)}\hat\phi(x(J)).
\]
Obviously $\var_\phi^\delta(x)\leq \var_\phi(x)$. We may also define
\[
\var_\phi^0(x) = \lim_{\delta\to 0^+} \var_\phi^\delta(x).
\]
One can easily see that the family of all $x\colon I\to\R$ satisfying $\var_\phi^0(x) = 0$ is a linear subspace of $\BV(\phi)$.

The following result is actually \cite[Proposition 5]{PW1} rephrased to the language of submeasure-generated variation spaces. Note that the assumption $\phi(\N)=\infty$ is actually equivalent to $\BV(\phi)\neq B(I)$ (see \cite[Remark 3.13]{GKT}). As we will see in Proposition \ref{prop:ext2}, the second assumption (i.e. $\lim_{m\to\infty} \phi(\{m\})=0$) guarantees that there are non-constant functions in $\BV_0(\phi)$.

\begin{theorem}\label{th:var0:equiv}
Assume that $\phi$ is a non-pathological lsc submeasure satisfying $\phi(\N)=\infty$ and $\lim_{m\to\infty} \phi(\{m\})=0$. Then for every $x\in \BV(\phi)$ we have
\[
\var_\phi^0(x) = \lim_{m\to +\infty} \var_{\phi_{(m)}}(x).
\] 
In particular, 
\[
\BV_0(\phi)=\{x\in \BV(\phi): \var_\phi^0(x)=0\}.
\]
\end{theorem}

\begin{proof}
First let us consider, for a fixed $\delta>0$, the family $C_\delta(x)$ consisting of all finite sequences $(I_1,I_2,\ldots,I_k)$ of nonoverlapping intervals such that for each $n\leq k$ we have $|x(I_n)|\geq\delta$. Let us denote 
\[
K(x,\delta)=\sup\{k:\text{there is }(I_1,I_2,\ldots,I_k)\in C_\delta(x)\}.
\]
It is easy to observe that for each $(I_1,I_2,\ldots,I_k)\in C_\delta(x)$ we have
\[
\phi(\{1,...,k\}) \leq \frac{\var_\phi(x)}{\delta}.
\]
Since $\phi(\N)=\infty$, $K(x,\delta)$ is finite.

Let us now take any sequence of nonoverlapping subintervals of $I$. Then for all but at most $K(x,\delta)$ of them we may say that $|x(J)|<\delta$. This leads to the inequality
\[
\var_{\phi_{(m)}}(x) \leq 2\Vert x\Vert_\infty \sup\{\phi(C): |C|=K(x,\delta)\wedge \min C>m \} + \var_\phi^\delta(x).
\]
Passing with $m$ to $\infty$ and using the assumption $\lim_{m\to\infty} \phi(\{m\})=0$ (which gives us 
\[
\lim_{m\to +\infty} \sup\{ \phi(C) : |C|=n \wedge \min C>m\} = 0.
\]
for all $n\in\N$), we have 
\[
\lim_{m\to+\infty} \var_{\phi_{(m)}}(x) \leq \var_\phi^\delta(x).
\]
Since $\delta$ was arbitrary, we get
 \[
\lim_{m\to+\infty} \var_{\phi_{(m)}}(x) \leq \var_\phi^0(x).
\]  

On the other hand, for any $\delta>0$ and $m\in\N$ we have
\[
\var_\phi^\delta(x) \leq \delta \phi(\{1,\ldots,m\}) + \var_{\phi_{(m)}}(x).
\]
Passing to the limit with $\delta\to 0^+$ we have
\[
\var_\phi^0(x) \leq  \var_{\phi_{(m)}}(x)
\]
for all $m\in\N$, which leads to 
\[
\var_\phi^0(x) \leq \lim_{m\to+\infty}  \var_{\phi_{(m)}}(x)
\]
and completes the proof.
\end{proof}

\section{Extreme cases}

Recall that $\phi(\N)<+\infty$ is equivalent to $\BV(\phi)=B(I)$, the space of all bounded functions (see \cite[Remark 3.13]{GKT}). The shape of $\BV_0(\phi)$ requires additional investigation:

\begin{proposition}
\label{prop:ext1}
Let $\phi$ be a non-pathological lsc submeasure. Then 
$\BV_0(\phi) = B(I)$ if and only if 
\[\lim_{n\to\infty} \phi(\N\setminus \{1,\ldots,n\})=0.\]
\end{proposition}

\begin{proof}
Assume first that $\lim_{n\to\infty} \phi(\N\setminus \{1,\ldots,n\})=0$. The inclusion $\BV_0(\phi) \subseteq B(I)$ is obvious, so we only need to show the opposite one. Fix any $x\in B(I)$. Then there is $M>0$ such that $|x(J)|<M$ for all intervals $J\subseteq I$. Observe that
    \[
    \var_{\phi_{(m)}}(x) \leq M\cdot \phi(\N\setminus \{1,\ldots,m\})\to 0.\]
    Hence $x\in \BV_0(\phi)$.

Assume now that $\lim_{n\to\infty} \phi(\N\setminus \{1,\ldots,n\})\neq 0$. Define $x\in B(I)$ by
    \begin{equation*}
        x(t)=\begin{cases}
            1, & \text{if }t=\frac{1}{n}\text{ for some }n\in\N,\\
            0, & \text{if }t\in[0,1]\setminus\{\frac{1}{n}:n\in\N\}.
        \end{cases}
    \end{equation*}
Then
    \[
    \var_{\phi_{(m)}}(x)=1\cdot \phi(\N\setminus \{1,\ldots,m\})\not\to 0.\]
    Hence $x\notin \BV_0(\phi)$.
\end{proof}

\begin{proposition}\label{prop:ext2}
Let $\phi$ be a non-pathological lsc submeasure. Then
$\BV_0(\phi)$ consists only of constant functions if and only if 
\[\limsup_{n\to\infty}\phi(\{n\})>0.\]
\end{proposition}

\begin{proof}
Assume first that $\limsup_{n\to\infty}\phi(\{n\})=\alpha>0$. Let $x\in B(I)$ be a non-constant function. Then there exists $[a,b]\subset I$ such that $|x(b)-x(a)|=c>0$. Fix any $m\in\N$. Then we can find $n>m$ such that $\phi(\{n\})>\frac{\alpha}{2}$.  Let $\mu\in\mathcal{M}_{\phi_{(m)}}$ be given by $\mu(\{i\})=0$ for $i\not= n$ and $\mu(\{n\})=\phi(\{n\})$.  Then
    \[
    \var_{\phi_{(m)}}(x) \geq \mu(\{n\})|x([a,b])| > \frac{\alpha}{2} c.    \]
    Since $m$ was arbitrary, we can conclude that $x\notin \BV_0(\phi)$.

    Assume now that $\limsup_{n\to\infty}\phi(\{n\})=0$. Let $x\in B(I)$ be given by
    \begin{equation*}
        x(t)=\begin{cases}
            1, & \text{if }t=1,\\
            0, & \text{if }t\in[0,1).
        \end{cases}
    \end{equation*}
    Then $x$ is not constant and belongs to $\BV_0(\phi)$, since
    \[
    \var_{\phi_{(m)}}(x)=1\cdot \sup_{n>m}\phi(\{n\})\to 0.
    \]
\end{proof}

\section{Basic properties}

Note that for any permutation $\pi\colon\N\to\N$ the submeasure $\psi$ given by $\psi(C)=\phi(\pi[C])$ generates the same bounded variation space: $\BV(\psi)=\BV(\phi)$ (see \cite[Remark 3.15]{GKT}).

\begin{proposition}
    Let $\pi\colon\N\to\N$ be any permutation. Then the submeasure $\psi$ given by $\psi(C)=\phi(\pi[C])$ generates the same space of functions continuous in bounded variation, i.e. $\BV_0(\psi)=\BV_0(\phi)$.
\end{proposition}

\begin{proof}
Fix a permutation $\pi\colon\N\to\N$, $x\in \BV_0(\psi)$ and $\varepsilon>0$. Let $m_\psi\in\N$ be such that $\var_{\psi_{(m)}}(x)<\varepsilon$ for all $m>m_\psi$. Define $m_\phi=\max \pi[\{1,2,\ldots,m_\psi\}]$. We will show that $\var_{\phi_{(m)}}(x)<\varepsilon$ for all $m>m_\phi$.

Fix any $m>m_\phi$, $J=(J_1,J_2,\ldots)\in\mathcal{P}_I$ and $\mu\in\mathcal{M}_{\phi_{(m)}}$. Define $\hat{J}_n=J_{\pi(n)}$ for all $n\in\mathbb{N}$. Let also $\hat{\mu}$ be the measure given by $\hat{\mu}(\{n\})=\mu(\{\pi(n)\})$ for all $n\in\mathbb{N}$. Observe that $\hat{\mu}\in\mathcal{M}_{\psi}$, since
\begin{equation*}
\begin{split}
\hat{\mu}(C) & =\mu(\pi[C])\leq\phi_{(m)}(\pi[C])\\
          & =\phi(\pi[C]\cap\{m+1,m+2,\ldots\}) \\
          & \leq\phi(\pi[C\cap\{m_\psi+1,m_\psi+2,\ldots\}]) \\
          & =\psi(C\cap\{m_\psi+1,m_\psi+2,\ldots\})=\psi_{(m_\psi)}(C)
\end{split}
\end{equation*}
    for all $C\subseteq\mathbb{N}$, since 
    \begin{equation*}
\begin{split}
\pi[C]\cap\{m+1,m+2,\ldots\} & \subseteq \pi[C]\cap\{m_\phi+1,m_\phi+2,\ldots\} \\
          & \subseteq\pi[C\cap\{m_\psi+1,m_\psi+2,\ldots\}].
\end{split}
\end{equation*}
    Moreover, $\hat{\mu}(\{i\})=\mu(\{\pi(i)\})=0$ for all $i\leq m_{\psi}$ (by $m>m_\phi\geq \pi(i)$ and $\mu\in\mathcal{M}_{\phi_{(m)}}$; see Remark \ref{rem:M_phi_m}). Hence $\hat{\mu}\in\mathcal{M}_{\psi_{(m)}}$ by Remark \ref{rem:M_phi_m}.
    
    Note that
    \begin{equation*}
\begin{split}
\sum_{n\in\mathbb{N}}\mu(\{n\})|x(J_n)| & =\sum_{n>m}\mu(\{n\})|x(J_n)|\leq \sum_{n>m_\phi}\mu(\{n\})|x(J_n)| \\
          & \leq \sum_{k>m_\psi}\mu(\{\pi(k)\})|x(J_{\pi(k)})|=\sum_{k>m_\psi}\hat{\mu}(\{k\})|x(\hat{J}_k)| \\
          & \leq \var_{\psi_{(m)}}(x)<\varepsilon.
\end{split}
\end{equation*}
    This shows that $\var_{\phi_{(m)}}(x)<\varepsilon$ for all $m>m_\phi$. Hence $\BV_0(\psi)\subseteq \BV_0(\phi)$

    Since the above argument works for arbitrary bijections and we have $\phi(C)=\psi(\pi^{-1}[C])$, by the above we can conclude that $\BV_0(\phi)\subseteq \BV_0(\psi)$. Thus $\BV_0(\phi)=\BV_0(\psi)$.
\end{proof}

There is also a general observation that the difference in the level of sequence spaces immediately implies that $\BV(\phi)$ and $\BV_0(\phi)$ are different. Before stating it, let us make one observation.

\begin{remark}
If $\lim_{n\to\infty} \phi(\N\setminus \{1,\ldots,n\})=0$ then $\BV_0(\phi)=\BV(\phi)$. Indeed, $\lim_{n\to\infty} \phi(\N\setminus \{1,\ldots,n\})=0$ implies $\phi(\N)<\infty$. Then we get $\BV_0(\phi)=B(I)=\BV(\phi)$ by Proposition \ref{prop:ext1} and \cite[Proposition 3.18]{GKT}.
\end{remark}

In the next result we will need the following condition:
\begin{equation}\label{eq}
\forall_{x,y\in\mathbb{R}^\N}\left(\left(\forall_{n\in\N}\sum_{i=1}^n |x_i|\leq\sum_{i=1}^n |y_i|\right)\implies \hat{\phi}(x)\leq\hat{\phi}(y)\right).
\end{equation}
It already appeared in the literature in \cite[Theorem 4.6]{GKT}. The measures given by Waterman sequences as well as simple density submeasures (for definition see the next Section) both satisfy this condition (see \cite[proofs of Theorems 5.1 and 6.2]{GKT}).

We will also need the following definition introduced in \cite[Definition 3.1]{GKT}.

\begin{definition}
$m\FIN(\phi)$ is the subspace of $\FIN(\phi)$ consisting of all sequences with nonincreasing absolute values while $m\EXH(\phi)$ is the subspace of $\EXH(\phi)$ consisting of all sequences with nonincreasing absolute values. 
\end{definition}

\begin{theorem}
\label{thm:FINneqEXH}
Let $\phi$ be a non-pathological lsc submeasure such that $\lim_{n\to\infty} \phi(\N\setminus \{1,\ldots,n\})\neq 0$ and it satisfies condition (\ref{eq}). If $m\FIN(\phi)\setminus \EXH(\phi)$ then $\BV_0(\phi) \neq \BV(\phi)$.
\end{theorem}

\begin{proof}
Assume first that $\phi(\N)<\infty$. Then $\BV(\phi)=B(I)$ (by \cite[Proposition 3.18]{GKT}), while $\BV(\phi)\neq B(I)$ (by $\lim_{n\to\infty} \phi(\N\setminus \{1,\ldots,n\})\neq 0$ together with Proposition \ref{prop:ext1}). Hence $\BV_0(\phi) \neq \BV(\phi)$.

Assume now that $\phi(\N)=\infty$. Let $r=(r_n)_{n\in\N}\in m\FIN(\phi)\setminus \EXH(\phi)$. Define $x:I\to\mathbb{R}$ to be the piecewise linear function such that $x(1)=0$, $x(0)=\sum_{n=1}^\infty (-1)^{n+1}|r_n|$ and $x(\frac{1}{2^k})=\sum_{n=1}^k (-1)^{n+1}|r_n|$ for all $k\in\N$. Then $x$ is well-defined.

Let $\hat{J}=(\hat{J}_n)_{n\in\N}\in\mathcal{P}_I$ be the sequence of intervals given by $\hat{J}_1=[\frac{1}{2},1]$ and $\hat{J}_{k+1}=[\frac{1}{2^{k+1}},\frac{1}{2^{k}}]$ for all $k\in\N$. Then $|x(\hat{J}_k)|=|r_k|$ for all $k\in\N$. 

Observe that $x\notin \BV_0(\phi)$, since 
\[
\sup_{J\in\mathcal{P}_I}\hat{\phi}_{(m)}(x(J))\geq\hat{\phi}_{(m)}(x(\hat{J}))=\hat{\phi}_{(m)}(r)\not\to 0.
\]
On the other hand, by monotonicity of $(|r_n|)_{n\in\omega}$, for each $J\in\mathcal{P}_I$ we have
\[
\forall_{n\in\N}\sum_{i=1}^n |x(J_i)|\leq\sum_{i=1}^n |x(\hat{J}_i)|.
\]
Indeed, this is actually a simple observation, which may either be proved directly or deduced from the general observation (see \cite[Proposition 1.1]{PW-Nachrichten}) stating that when selecting the ends of the intervals we should select {\em points of varying monotonicity} to get higher value of the sum. Hence, by the assumption imposed on $\phi$, we have $\sup_{J\in\mathcal{J}_I}\hat{\phi}(x(J))=\hat{\phi}(x(\hat{J}))=\hat{\phi}(r)<\infty$.
\end{proof}

\section{Some thoughts on simple density submeasures}

In this Section we are interested in lsc submeasures of the form $\phi_g(C)=\sup_{n\in\N}\frac{|C\cap\{1,2,\ldots,n\}|}{g(n)}$, where $g:\N\to(0,\infty)$ satisfies the following conditions:
\begin{itemize}
    \item[(a)] $(g(n))_{n\in\N}$ is non-decreasing;
    \item[(b)] $\lim_n g(n)=\infty$;
    \item[(c)] $(\frac{n}{g(n)})_{n\in\N}$ does not tend to zero;
    \item[(d)] $(\frac{n}{g(n)})_{n\in\N}$ is non-decreasing.
\end{itemize}
Ideals of the form $\Exh(\phi_g)$ (for functions $g$ satisfying all the above requirements except the last one) have been extensively studied in \cite{simpleden3}, \cite{simpleden2} and \cite{simpleden1}. For that reason, we decided to write the last item separately despite the fact that (c) follows from (d). We will denote by $\mathcal{G}$ the set of all functions $g$ satisfying conditions (a)-(d). Note that in this case we have
$$\hat{\phi}_g(x)=\sup_{n\in\N}\frac{\sum_{i=1}^n|x_i|}{g(n)}$$
and
\[
\phi_{g,(m)}(C) = \sup_{n> m}\frac{|C\cap\{m+1,\ldots,n\}|}{g(n)},
\]
which leads to 
\[
\hat{\phi}_{g,(m)}(x) = \sup_{n> m}\frac{\sum_{i=m+1}^n|x_i|}{g(n)}.
\]

Then, for a bounded function $x\colon I\to\R$ there is
\begin{equation*}
\begin{split}
   \var_{\phi_{g,(m)}}(x) &  = \sup_{J\in\cP_I} \sup_{n>m}\frac{1}{g(n)}\sum_{i=1}^{n-m}|x(J_i)| = \sup_{J\in\cP_I} \sup_{n\in\N}\frac{1}{g(n+m)}\sum_{i=1}^{n}|x(J_i)| \\
   & =\sup_{J\in\cP_I} \sup_{n\in\N} \frac{1}{g(n)}\sum_{i=1}^{n}|x(J_i)|\frac{g(n)}{g(n+m)}.
\end{split}
\end{equation*}
This may be estimated by
\[
\var_{\phi_{g,(m)}}(x) \leq \sup_{n\in\N}\frac{g(n)}{g(n+m)} \var_{\phi_g}(x),
\]
which leads to the obvious conclusion that if
\[
\lim_{m\to +\infty} \sup_{n\in\N} \frac{g(n)}{g(n+m)} = 0,
\]
then $\BV_0(\phi_g) = \BV(\phi_g)$. However, by the following Remark, the above condition never holds. As we will see, it is a consequence of a much stronger observation: it appears that for any $g\in\cG$ we have $\BV_0(\phi_g) \neq \BV(\phi_g)$ (see Theorem \ref{thm:simple_density}). 

\begin{remark}
Let $g\in\mathcal{G}$ and $m\in\N$. Then $\sup_{n\in\N} \frac{g(n)}{g(n+m)} = 1$. Indeed, since $(\frac{n}{g(n)})_{n\in\N}$ is non-decreasing, we have
\[
\frac{n}{g(n)} \leq \frac{n+m}{g(n+m)},
\]
so using condition (a) we get
\[
\frac{n}{n+m} \leq \frac{g(n)}{g(n+m)} \leq 1.
\]
This means that $\sup_{n\in\N} \frac{g(n)}{g(n+m)} = 1$.
\end{remark}

\begin{lemma}
\label{lem:simple_density}
For each $g\in\cG$ satisfying $\lim_{n\to\infty}\frac{n}{g(n)}=\infty$ we have $m\FIN(\phi_g) \setminus \EXH(\phi_g)\neq\emptyset$. 
\end{lemma}

\begin{proof}
We will find a non-increasing sequence of positive reals $r\in \FIN(\phi_g) \setminus \EXH(\phi_g)$. 

We need to inductively construct sequences $(n_k)_{k\in\N_0}\in\N_0^{\N_0}$, $(s_k)_{k\in\N_0}\in(0,\infty)^{\N_0}$ and $(r_n)_{n\in\N}\in(0,\infty)^\N$ satisfying:
\begin{itemize}
    \item[(i)] $n_0=0$ and $s_0=g(1)$,
    \item[(ii)] $s_{k+1}<s_k$ and $n_{k+1}>n_k$ for all $k\in\N_0$,
    \item[(iii)] $r_n=s_{k+1}$ for all $n\in(n_k,n_{k+1}]$ and $k\in\N_0$,
    \item[(iv)] $\frac{\sum_{j=1}^{n}r_j}{g(n)}< 2$ for all $n\in\N$,
    \item[(v)] $\frac{\sum_{j=1}^{n_k}r_j}{g(n_k)}\geq 1$ for all $k\in\N$.
\end{itemize}
In the initial step just put $n_0=0$ and $s_0=g(1)$. Suppose now that all $n_i$ and $s_i$ for $i\leq k$ as well as all $r_j$ for $j\leq n_k$ are already defined. Pick any $s_{k+1}\in(0,\infty)$ satisfying $s_{k+1}<s_k$ and 
\[
s_{k+1}<2g(n_k+1)-\sum_{j=1}^{n_k}r_j.
\]
(Note that by (iv) $2g(n_k+1)-\sum_{j=1}^{n_k}r_j>0$.)
Denote 
\[
a^{(k)}_l=\frac{(\sum_{j=1}^{n_k}r_j)+l\cdot s_{k+1}}{g(n_k+l)}
\]
for all $l\in\N$. Let $n_{k+1}\in\N$ be minimal with the properties $n_{k+1}>n_k$ and $a^{(k)}_{n_{k+1}-n_k}\geq 1$. Such $n_{k+1}$ exists, since $\lim_{l\to\infty}g(n_k+l)=\infty$ together with
\[
\lim_{l\to\infty}\frac{l}{g(n_k+l)}=\lim_{l\to\infty}\frac{n_k+l}{g(n_k+l)}-\lim_{l\to\infty}\frac{n_k}{g(n_k+l)}=\infty
\]
imply that $\lim_{l\to\infty}a^{(k)}_l=\infty$. Define $r_n=s_{k+1}$ for all $n\in(n_k,n_{k+1}]$ and observe that items (ii) and (iii) are satisfied. In order to verify (v), note that
\[
\frac{\sum_{j=1}^{n_{k+1}}r_j}{g(n_{k+1})}=a^{(k)}_{n_{k+1}-n_k}\geq 1.
\]
It remains to show (iv) for $n\in(n_k,n_{k+1}]$. Observe that $\frac{\sum_{j=1}^{n}r_j}{g(n)}=a^{(k)}_{n-n_k}$ for all $n\in(n_k,n_{k+1}]$. Thus, it suffices to show that $a^{(k)}_l\leq 2$ for $l=1,\ldots,n_{k+1}-n_k$. By the second condition imposed on $s_{k+1}$ we get $a^{(k)}_1<2$. Moreover,
\begin{equation*}
\begin{split}
    a^{(k)}_{l+1}-a^{(k)}_l & =\frac{(\sum_{j=1}^{n_k}r_j)+l\cdot s_{k+1}+s_{k+1}}{g(n_k+l+1)}-\frac{(\sum_{j=1}^{n_k}r_j)+l\cdot s_{k+1}}{g(n_k+l)} \\
    & =\frac{(\sum_{j=1}^{n_k}r_j)+l\cdot s_{k+1})(g(n_k+l)-g(n_k+l+1))}{g(n_k+l)\cdot g(n_k+l+1)} \\
    & +\frac{s_{k+1}\cdot g(n_k+l)}{g(n_k+l)\cdot g(n_k+l+1)} \\
    & \leq \frac{s_{k+1}\cdot g(n_k+l)}{g(n_k+l)\cdot g(n_k+l+1)}=\frac{s_{k+1}}{g(n_k+l+1)}<\frac{s_{0}}{g(1)}=1
\end{split}
\end{equation*}
for all $l\in\N$ (by the fact that $(g(n))_{n\in\N}$ is non-decreasing). Thus, by the definition of $n_{k+1}$, 
\[
a^{(k)}_{l+1}<1+a^{(k)}_{l}<1+1=2
\]
for all $l<n_{k+1}-n_k$, which finishes the construction.

Once all $r_n$ are defined, observe that $r=(r_n)_{n\in\N}$ is a non-increasing sequence of positive reals. Moreover, $r\in \FIN(\phi_g)$, since $\sup_{n\in\N}\frac{\sum_{i=1}^n r_i}{g(n)}\leq 2$ by item (iv). On the other hand, $r\notin \EXH(\phi_g)$ as for any $m\in\N$ we get
\[
\sup_{n>m}\frac{\sum_{i=m+1}^n r_i}{g(n)}=\sup_{n>m}\left(\frac{\sum_{i=1}^n r_i}{g(n)}-\frac{\sum_{i=1}^m r_i}{g(n)}\right)\geq 1
\]
by item (v) and $\lim_{k\to\infty} g(n_k)=\infty$.
\end{proof}

\begin{theorem}
\label{thm:simple_density}
For each $g\in\cG$ we have $\BV_0(\phi_g) \neq \BV(\phi_g)$.
\end{theorem}

\begin{proof}
By condition (d), $\lim_{n\to\infty}\frac{n}{g(n)}$ exists. Assume first that it is equal to some $\alpha<\infty$ (note that $\alpha>0$ by condition (c)). Then $\phi_g(\N)=\sup_{n\in\N}\frac{n}{g(n)}=\alpha<\infty$, so $\BV(\phi_g)=B(I)$ by \cite[Proposition 3.18]{GKT}. On the other hand, 
\[
\lim_{n\to\infty} \phi_g(\N\setminus\{1,\ldots,n\})=\lim_{n\to\infty}\sup_{k>n}\frac{k-n}{g(k)}=\alpha>0.
\]
Hence, $\BV(\phi_g)\neq B(I)$ by Proposition \ref{prop:ext1} and we can conclude that $\BV_0(\phi_g) \neq \BV(\phi_g)$.

Assume now that $\lim_{n\to\infty}\frac{n}{g(n)}=\infty$. Then Lemma \ref{lem:simple_density} gives us $r\in m\FIN(\phi_g) \setminus \EXH(\phi_g)$. Since $\lim_{n\to\infty} \phi_g(\N\setminus\{1,\ldots,n\})=\infty$ (for the same reason as in the previous paragraph) and $\phi_g$ satisfies condition (\ref{eq}) (by \cite[proof of Theorem 5.1]{GKT}), Theorem \ref{thm:FINneqEXH} finishes the proof.
\end{proof}

\section{Other classes of submeasures}\label{sec:other:submeasures}

The general condition determining whether $x$ belongs to $\BV(\phi)$ or not,
\[
\sup_{J\in\cP_I} \sup_{\mu\in\cM_\phi} \sum_{k\in\N} \mu(\{k\})|x(J_k)| = \sup_{\mu\in\cM_\phi} \sup_{J\in\cP_I}  \sum_{k\in\N} \mu(\{k\})|x(J_k)| < +\infty,
\]
actually leads us to the observation that 
\[
\BV(\phi) \subseteq \bigcap_{\mu\in\cM_\phi} \BV(\mu).
\]

It appears that in some cases we may characterize $\BV(\phi)$ in terms of $\BV(\mu)$, where $\mu$ are measures. 

\begin{example}
Let $\N=I_1\cup I_2$, where $I_1=\{ 2k-1 : k\in\N\}$ and $I_2 = \{2k: k \in\N\}$. Let us take two Waterman sequences $A=(a_n)$ and $B=(b_n)$ and define the submeasure $\phi_{A,B}$ by
\[
\phi_{A,B}(C)= \max\left\{ \sum_{k\in C\cap I_1} a_{\frac{k+1}{2}}, \sum_{k\in C\cap I_2} b_{\frac{k}{2}} \right\}
\]
for $C\subseteq\N$. We can show that $\BV(\phi_{A,B}) = ABV\cap BBV$ with the norm equal to
\[
\Vert x\Vert_{\BV(\phi_{A,B})} = \max\{ \Vert x\Vert_{\ABV}, \Vert x\Vert_{\BBV} \}.
\]
This norm is obviously equivalent to the norm given by
\[
 \Vert x\Vert_{\ABV} + \Vert x\Vert_{\BBV}. 
\]
Then by \cite[Theorem 6]{Per} we know that 
\[
\BV(\phi_{A,B}) = \BV(\phi_{C})
\]
is actually a Waterman space induced by the sequence $C=A+B$.
This means that $\BV(\phi_{A,B})=\BV_0(\phi_{A,B})$ if and only if the Shao\--Sablin index of $A+B$ is strictly smaller than $2$.

The interesting observation is that $\ABV\neq \BBV$ implies that $\FIN(\phi_{A,B})\neq \FIN(\phi_C)$. Without the loss of generality let us assume that $\ABV\setminus\BBV\neq\emptyset$. Then, by Theorem 6.2. from \cite{GKT} there must be $m\FIN(\phi_A)\setminus m\FIN(\phi_B)\neq \emptyset$. Hence there exists a sequence $x\in\FIN(\phi_A)\setminus\FIN(\phi_B)$ such that $|x_n|$ is nonincreasing.  Now we can use it to build a sequence $x'$ supported on $I_1$ (i.e. such that $x'_k=0$ for $k\in I_2$) in such way that $x'$ belongs to $\FIN(\phi_{A,B})$ but not to $\FIN(\phi_C)=\FIN(\phi_A)\cap\FIN(\phi_B)$. 

Indeed, take $x\in m\FIN(\phi_A)\setminus m\FIN(\phi_B)$. Then $\sum_{k\in\N} |x_k|a_k<\infty$ and $\sum_{k\in\N} |x_k|b_k=\infty$. Then, since both $|x_n|$ and $b_n$ are nonincreasing, there must be $\sum_{k\in\N}|x_{2k-1}|b_{2k-1}=\infty$.
Now let us  put $x'_k=x_k$ for $k\in I_1$ and $x'_k=0$ for $k\in I_2$. Then we get that $x'\not\in\FIN(\phi_B)$ as 
$$\sum_{k\in\N}|x'_k|b_k= \sum_{k\in\N}|x_{2k-1}|b_{2k-1}=\infty.$$
On the other hand, $x'\in\FIN(\phi_{A,B})$ since $\sum_{k\in I_2}|x'_k|b_{\frac{k}{2}}=0$ while $\sum_{k\in I_1}|x'_k|a_{\frac{k+1}{2}}$ is equal to 
$$\sum_{k\in\N} |x_{2k-1}|a_k, $$
which is smaller than $\sum_{k\in\N} |x_{k}|a_{k}<\infty$.

This shows that if $\ABV\neq\BBV$ there must be $\FIN(\phi_{A,B})\neq\FIN(\phi_C)$.

From the observation above it follows that in the level of bounded variation spaces we can get the same space from essentially different submeasures.
\end{example}

This may raise a question if we may say anything about the Shao-Sablin index of the sequence $A+B$ based on the indexes of sequences $A$ and $B$. We are going to show that there is a measure given by a Waterman sequence $A$ with $S(\phi_A)=2$ which is dominated by a measure given by a Waterman sequence $B$ with $S(\phi_B)<2$. Hence, it is not enough to show that $\max\{S_A,S_B\}=2$ to be sure that $S_{A+B}=2$.
\begin{example}
We will construct a proper Waterman sequence $W$ with some interesting properties. It will be constantly equal to $\frac{1}{\sqrt{2l_n}}$ on each segment $I_n=\{2l_{n-1}+1,...,2l_{n}\}$ for $i=1,2,...$, where $l_n$ are selected according to the procedure described below.
\newline
Start with $l_0=0$ and $l_1=1$. Then $I_1=\{1,2\}$ and $w_1=w_2 = \frac{1}{\sqrt{2}}$. Now, having the sequence $l_i$ defined for $i\leq n$, we also have $w_k$ defined for $k\in\bigcup_{i\leq n} I_i = \{1,....,2 l_n\}$ (namely, $w_k=\frac{1}{\sqrt{2l_i}}$ for $k\in I_i$). We are going to find a value $l_{n+1}\in\N$ in such a way that 
\begin{itemize}
    \item[(i)] $l_{n+1}> 2l_n$;
    \item[(ii)] $1 \le W_{l_{n+1}}\le 1 + \frac{1}{n}$, where
    \[
    W_j = \frac{\sum_{k=1}^{2l_n} w_k+(j-2l_n)\frac{1}{\sqrt{2j}}}{\sqrt{\frac{j}{2}}}
    \]
    for $j>2l_n$.
\end{itemize}
This is possible, because for any $j> 2l_n$ we have
\[
\sum_{k=1}^{2l_n} w_k+(j-2l_n)\frac{1}{\sqrt{2j}} = K_n + \sqrt{\frac{j}{2}} - \frac{\sqrt{2}l_n}{\sqrt{j}},
\]
where $K_n = \sum_{k=1}^{2l_{n}} w_k$, which means that 
\[
 W_{j} = 1 + \frac{K_n - \frac{\sqrt{2}l_n}{\sqrt{j}}}{\sqrt{\frac{j}{2}}}.
\]
We can see that for fixed $K_n$ and $l_n$ and $j$ big enough we have $K_n - \frac{\sqrt{2}l_n}{\sqrt{j}}>0$ and $\frac{K_n}{\sqrt{\frac{j}{2}}}<\frac{1}{n}$. Thus, we can find $l_{n+1}$ satisfying conditions (i) and (ii). 
\newline
Put $w_k=\frac{1}{\sqrt{2l_{n+1}}}$ for all $k\in I_{n+1}=\{2l_n+1,\ldots,2l_{n+1}\}$. Observe that 
    \[
    W_{l_{n+1}} = \frac{\sum_{k=1}^{l_{n+1}} w_k}{\sum_{k=l_{n+1}+1}^{2l_{n+1}} w_k}
    \]
since 
\[
\sum_{k=l_{n+1}+1}^{2l_{n+1}} w_k = l_{n+1}\frac{1}{\sqrt{2l_{n+1}}} = \sqrt{\frac{l_{n+1}}{2}}
\]
and
\[
\sum_{k=1}^{l_{n+1}} w_k =\sum_{k=1}^{2l_{n}} w_k + \sum_{k=2l_n+1}^{l_{n+1}} \frac{1}{\sqrt{2l_{n+1}}} = \sum_{k=1}^{2l_n} w_k+(l_{n+1}-2l_n)\frac{1}{\sqrt{2l_{n+1}}}.
\]
\newline
Given that 
\[
\lim_{n \to\infty} \frac{\sum_{k=1}^{l_{n+1}} w_k}{\sum_{k=l_{n+1}+1}^{2l_{n+1}} w_k}=1,\]
the sequence $W$ thus constructed has the Shao-Sablin index equal to $2$, and yet is majorized by the sequence $u_n=\frac{1}{\sqrt{n}}$, the Shao-Sablin index of which equals $\sqrt{2}$. Moreover, we can see that $W$ is a Waterman sequence (as $\sum_{i=1}^{2l_n} w_i \geq \sum_{i=1}^{2l_n} \frac{1}{\sqrt{2l_n}}=\sqrt{2l_n}$) and it is proper (as $\lim_{k\to\infty}w_k=0$).
\end{example}

The next example refers to much more complex situation where we have a submeasure $\phi$ which is not a measure, it is generated by an infinite sequence of submeasures,  but for which we know that there is $\EXH(\phi)=\FIN(\phi)$.

\begin{example}\label{ex:JT}    
Following the example given in Farah's book \cite[Example~1.11.1]{farah-book}, let us consider a nonpathological lsc submeasure for which $\EXH(\phi)=\FIN(\phi)$ but which is not induced by a Waterman sequence. It appears we can observe some interesting properties of the associated $\BV(\phi)$ space. Let 
us consider the disjoint intervals in $\N$ given by $I_n=[2^n, 2^{n+1})$ and a related submeasure $\phi_n(C)=\frac{1}{n^2} \min\{|C\cap I_n|, n\}$ for $C\subset\N$. Then let us define a submeasure 
\[
\phi(C)=\sum_{n\in\N}\phi_n(C).
\]
\noindent
Let us take any measure $\mu\in\cM_{\phi}$. Then $\mu_n = \mu|_{I_n}$ (i.e. measure $\mu$ truncated to the set $I_n$) satisfies 
\begin{itemize}
    \item $\mu_n(I_n)\leq \frac{1}{n}$ (without the loss of generality we may assume that $\mu_n(I_n)= \frac{1}{n}$);
    \item $\mu_n(\{k\})\leq \frac{1}{n^2}$ for all $k\in I_n$.
\end{itemize}

Let us now take a bounded function $x\colon I\to\R$ and any partition $J=(J_n)\in\mathcal{P}_I$ of the domain $I$, and the corresponding sequence $(|x(J_n)|)_{n\in\N}$. We can see that 
\[
\sup_{\mu\leq \phi} \sum_{k\in\N} \mu(\{k\})|x(J_k)| = \sup_{\mu\leq \phi} \sum_{n\in\N} \sum_{k\in I_n} \mu_n(\{k\})|x(J_k)|. 
\]
We can easily see that for each fixed $I_n$ the biggest value of the sum $\sum_{k\in I_n} \mu_n(\{k\})|x(J_k)|$ is achieved when the measure $\mu_n$ admits the highest possible values on singletons and when it is supported on such $n$-element subset $C_n\subset I_n$ that it corresponds to the indices from $I_n$ with $n$ highest values of $|x(J_k)|$. Hence, instead of taking the supremum over all measures $\mu\in\cM_\phi$, we may take supremum over measures $\mu$, which truncated to $I_n$ are supported on $n$-element subset $C_n\subset I_n$ and attain value $\frac{1}{n^2}$ on the atoms in the support. Let us denote this family of measures as $\cM_\phi'$.

We can see that for a bounded function $x$
\[
\var_\phi(x) = \sup_{J\in \mathcal{P}_I}\sup_{\mu\in\cM_\phi} \sum_{k\in\N} \mu(\{k\})|x(J_k)| = \sup_{J\in \mathcal{P}_I}\sup_{\mu\in\cM_\phi'} \sum_{k\in\N} \mu(\{k\})|x(J_k)|=
\]
\[
\sup_{\mu\in\cM_\phi'} \sup_{J\in \mathcal{P}_I}\sum_{k\in\N} \mu(\{k\})|x(J_k)|.
\]

Assume now that $x\in\BV(\phi)$. We can see that for any  partition $J\in \mathcal{P}_I$ the sequence $(|x(J_k)|)_{k\in\N}$ satisfies $\lim_{k\to+\infty}|x(J_k)| = 0$. 
If there are infinitely many intervals for which $|x(J_k)|>0$ we may remove such  intervals $J_k$ that $|x(J_k)|=0.$ This means that for any situation we can find the permutation $\sigma\colon\N\to\N$ such that $|x(J_{\sigma(k)})|\geq |x(J_{\sigma(k+1)})|$. 

Now we can see that for each of the measures $\mu\in\cM_\phi'$ and any partition $J\in \mathcal{P}_I$ we may add degenerated intervals in such a way that we have a~new partition $(J'_k)_{k\in\N}$ such that for any $k\in\N$ with $\mu(\{k\})=0$ we have $|x(J'_k)|=0$. Moreover, by the remark in the previous paragraph, we can assume that if $M\subset\N$ is the support of the measure $\mu$, then for $k,l\in M$ and $k\leq  l$ we have $|x(J'_k)|\geq |x(J'_l)|$. As we can see 
\[
\sum_{k\in\N} \mu(\{k\}) |x(J_k)| \leq \sum_{k\in\N} \mu(\{k\}) |x(J'_k)|.
\]
This means that for a fixed measure $\mu\in\cM_\phi'$ we have
\[
\sup_{J\in \mathcal{P}_I} \sum_{k\in\N} \mu(\{k\})|x(J_k)| = \sup_{J\in \mathcal{P}_I} \sum_{k\in\N} \mu_0(\{k\})|x(J_k)|
\]
where $\mu_0$ is a Waterman measure corresponding to the sequence
\[
A=\left(1,\frac{1}{4},\frac{1}{4},\frac{1}{9},\frac{1}{9},\frac{1}{9},\ldots \right).
\]

This means that $\BV(\phi)=ABV$, so the bounded variation space for a submeasure $\phi$ is actually a Waterman $ABV$ space. 

As we can see, the Shao-Sablin index for this sequence equals to $1$, hence $BV(\phi)=\BV_0(\phi)$, which by Theorem \ref{thm:FINneqEXH} above implies that $m\EXH(\phi)=m\FIN(\phi)$.
\end{example}

\begin{question}
Do there exist spaces $\BV(\phi)$ which are not Waterman spaces for such a submeasure $\phi$ that $m\EXH(\phi)=m\FIN(\phi)$ or even $\EXH(\phi)=\FIN(\phi)$? 
\end{question}

The previous example may be generalized to other sequences of sets $I_n$ covering $\N$ and different definitions of submeasures $\phi_n$ supported on $I_n$. 

\begin{example}\label{ex:JT2}   Let us take a family $I_n=[2^n, 2^{n+1})$ as in the Example \ref{ex:JT}.  
Let us define
\[
\phi_n(C) = a_n\min\{|C\cap I_n|, 1\}
\]
for $C\subseteq \N$ and fixed sequence $A=(a_n)_{n\in\N}\subset(0,+\infty)$ converging monotonically to $0$. Let us take the submeasure $\phi_A$ given by
\[
\phi_A(C) = \sum_{n\in\N} \phi_n(I_n\cap C).
\]

We can see that for ideals generated by a submeasure $\phi_A$ we have $\Exh(\phi_A)=\Fin(\phi_A)$. We can also easily check that for sequence spaces we have $\EXH(\phi_A)=\FIN(\phi_A)$.

As we can see (similarly as above), the space $\BV(\phi_A)$ is actually equal to the Waterman space $ABV$ generated by the sequence 
\[
(a_1,a_2,a_2,a_3,a_3,a_3,a_3,,a_4,\ldots).
\] 
This means that whether  we have $\BV(\phi_A)=\BV_0(\phi_A)$ or not depends on the Shao\--Sablin index of this sequence.
\end{example}

\section{Some thoughts on the Shao\--Sablin index}

As it was mentioned in the Subsection \ref{sec:cont:LBV}, in the case of Waterman spaces the fact whether $\Lambda \BV = \Lambda \BV_c$ is characterized by the value of the Shao\--Sablin index of the sequence $\Lambda$. We may ask if a similar indicator may be used for the more general case of  $\BV_0(\phi)$ spaces. 

\begin{definition}\label{Shao-Sablin-submeasure}
    Let $\phi$ be a non-pathological lsc submeasure. The value
    \[
    S(\phi) = \limsup_{n\to +\infty} \frac{\phi(\{1,\ldots,2n\})}{\phi(\{1,\ldots,n\})}
    \]
    is called the Shao-Sablin index of $\phi$. 
\end{definition}

In the general case it is difficult to control the value of this coefficient, but with an additional assumption (which does not look like an exotic one) we may say something:

\begin{remark}
\label{rem:SSindex}
    The assumption 
\begin{equation*}
    \phi(\{1,...,n\})\geq \sup\{ \phi(C) : C\subseteq\N \wedge |C|=n \}
\end{equation*}
guarantees that $S(\phi)\leq 2$.
\end{remark}

\begin{proposition}
\label{prop:SSindex}
For each $s\in[1,2]$ there exists $g\in\cG$ such that the Shao-Sablin index of the submeasure $\phi_g$ is equal to $s$.     
\end{proposition}
\begin{proof}
For $s\in[1,2)$, let $a=1-\log_2 s$. Then $2^{1-a}=s$ and $a\in (0,1]$. Put $g(n)=n^a$. It is easy to check that $g\in\cG$ (note that in the case of $a=0$ we would get $g(n)=1$, which does not satisfy (b)). Moreover, $\phi_g(\{1,\ldots,n\})=\frac{n}{g(n)}=n^{1-a}$ for all $n\in\N$. Hence, 
\[
S(\phi_g)=\limsup_{n\to\infty}\frac{\frac{2n}{(2n)^a}}{\frac{n}{n^a}}=\frac{2}{2^a}=2^{1-a}=s.
\]

For $s=2$, let $g(n)=\ln (n+1)$ for all $n\in\N$. Then $g\in\cG$ and $\phi_g(\{1,\ldots,n\})=\frac{n}{g(n)}=\frac{n}{\ln (n+1)}$ for all $n\in\N$. Moreover,
\[
\frac{\phi_g(\{1,\ldots,2n\})}{\phi_g(\{1,\ldots,n\})}=2\frac{\ln (n+1)}{\ln (2n+1)}\geq 2\frac{\ln (n+1)}{\ln (2n+2)}=2\frac{\ln (n+1)}{\ln (n+1)+\ln 2}\to 2.
\]
Since $S(\phi_g)\leq 2$ by Remark \ref{rem:SSindex}, we obtain $S(\phi_g)=2$.
\end{proof}

Summing up:
\begin{itemize}
\item For $\phi=\phi_A$, where $A$ is a Waterman sequence, the Shao-Sablin index recognizes between the situation $\BV_0(\phi)=\BV(\phi)$ and $\BV_0(\phi)\neq\BV(\phi)$ (by Theorem \ref{thr1}).
\item  For $\phi=\phi_g$, where $g\in\cG$, we have always $\BV_0(\phi)\neq\BV(\phi)$ (by Theorem \ref{thm:simple_density}), while the Shao-Sablin index can have any possible value in $[1,2]$ (by Proposition \ref{prop:SSindex}).
    \item Consequently, the Shao-Sablin index alone does not recognize between the situation $\BV_0(\phi)=\BV(\phi)$ and $\BV_0(\phi)\neq\BV(\phi)$.
    \item Examples \ref{ex:JT}  and \ref{ex:JT2}  from Section \ref{sec:other:submeasures} show that the Definition \ref{Shao-Sablin-submeasure} is actually useless: when we look at the value
\[
\frac{\phi(\sum_{k=1}^n I_k)}{\phi(\sum_{k=1}^{n-1} I_k)} = 1 + \frac{\phi_n(I_n)}{\sum_{k=1}^{n-1}\phi_k(I_k)}
\]
we can see that it tends to $1$ as $n\to+\infty$. But as we could see both $\BV(\phi)=\BV_0(\phi)$ and $\BV(\phi)\neq\BV_0(\phi)$ may happen. Which situation is actually there depends on the Shao\--Sablin index of some {\em effective} sequence $A$ such that $\BV(\phi)=ABV$.
\end{itemize}

\section*{Acknowledgment}
The work of Jacek Gulgowski was supported by the project ``Center for Trustworthy Artificial Intelligence for Life Sciences" implemented under the International Research Agendas programme of the Foundation for Polish Science co-financed by the European Union under the European Funds for Smart Economy 2021-2027 (FENG).

\bibliographystyle{amsplain}
\bibliography{references.bib}

\end{document}